\documentclass[preprint]{imsart}

\RequirePackage{amsthm,amsmath,amsfonts,amssymb}
\RequirePackage[numbers]{natbib}
\RequirePackage[colorlinks,citecolor=blue,urlcolor=blue]{hyperref}
\RequirePackage{graphicx}

\RequirePackage[OT1]{fontenc}
\usepackage{bbm}
\usepackage[numbers]{natbib}
\usepackage{amsthm,amsmath}
\usepackage{graphicx,ifthen}
\usepackage{latexsym}
\usepackage{mathrsfs}
\usepackage{amssymb}
\usepackage{hyperref}
\usepackage{xcolor}
\usepackage{amsfonts}
\usepackage{enumerate}
\usepackage{caption, subcaption}
\usepackage{algorithmicx}

\usepackage{algorithm,algpseudocode}

\startlocaldefs
\numberwithin{equation}{section}
\theoremstyle{plain}
\newtheorem{thm}{Theorem}[section]

\newtheorem{prop}{Proposition}[section]

\newtheorem{rem}{Remark}[section]

\def\Beta{\text{Beta}}

\newcommand{\Lower}[2]{\smash{\lower #1 \hbox{#2}}}
\newcommand{\ben}{\begin{enumerate}}
\newcommand{\een}{\end{enumerate}}
\newcommand{\bi}{\begin{itemize}}
\newcommand{\ei}{\end{itemize}}

\newcommand{\PitmanDiagram}[5]
{#1 \qquad
\begin{array}{c}
#2\\
\xrightarrow{#5}\\
\xleftarrow{#5}\\
#4
\end{array}
\qquad #3
}

\endlocaldefs

\begin{document}

\begin{frontmatter}
\title{Posterior distributions of Gibbs-type priors}
\runtitle{Gibbs-type Priors}

\begin{aug}
\author[A]{\fnms{Lancelot F.}~\snm{James}\ead[label=e1]{lancelot@ust.hk}
}
\address[A]{{Department of Information Systems, Business Statistics
and Operations Management}, The Hong Kong University of Science and Technology\printead[presep={,\ }]{e1}}

\end{aug}

\begin{abstract}
Gibbs type priors have been shown to be natural generalizations of Dirichlet process (DP) priors used for intricate applications of Bayesian nonparametric methods. This includes applications to mixture models and to species sampling models arising in populations genetics. Notably these latter applications, and also applications where power law behavior such as that arising in natural language models are exhibited, provide instances where the DP model is wholly inadequate. Gibbs type priors include the DP, the also popular Pitman-Yor process and closely related normalized generalized gamma process as special cases. However, there is in fact a richer infinite class of such priors, where, 
despite knowledge about the exchangeable marginal structures produced by sampling $n$ observations, descriptions of the corresponding posterior distribution, a crucial component in Bayesian analysis, remain unknown. This paper presents descriptions of the posterior distributions for the general class, utilizing a novel proof that leverages the exclusive Gibbs properties of these models. The results are applied to several specific cases for further illustration.
\end{abstract}

\begin{keyword}[class=AMS]
\kwd[Primary ]{60C05, 60G09}
\kwd[; secondary ]{60G57,60E99}
\end{keyword}

\begin{keyword}
\kwd{Posterior distributions,  Pitman Yor process, Gibbs partitions, Poisson Dirichlet distributions, stable Poisson-Kingman distributions}
\end{keyword}

\end{frontmatter}

\section{Introduction}
The classic problem in nonparametric statistics is the estimation of the unknown probability measure say $P,$ when $X_{1},\ldots,X_{n}$ are iid observations from $P,$ possibly taking values in some abstract Polish space $\Xi.$ The optimal solution in this case is known to be the simply expressed empirical measure $\hat{P}_{n}=\sum_{i=1}^{n}\delta_{X_{i}}/n.$  When $\Xi$ is the real line one may work with the empirical distribution function  $\hat{F}_{n}=\sum_{i=1}^{n}\mathbb{I}_{\{X_{i}\leq x\}}/n$ for $-\infty\leq x\leq \infty.$ Bayesian nonparametric statistics~(BNP) has its origins in providing a tractable solution to a Bayesian analogue of this problem. Ferguson's~\cite{Ferg73} seminal work on the Dirichlet process(DP) as a prior distribution, for $P,$ on spaces of probability measures provides the most tractable of such solutions yielding a posterior distribution of $P$ given $X_{1},\ldots,X_{n}$ which, as a natural extension of the Dirichlet-Multinomial problem, is also a Dirichlet process that concentrates its mass around $\hat{P}_{n}$ as $n\rightarrow\infty.$ While  this requires a much more complex framework than merely using $\hat{P}_{n},$ the posterior distribution encodes much more distributional information about various functionals of $P,$ and importantly its simplicity, in conjunction with developments in computational procedures and modeling flexibility in terms of  applications suitable for machine learning, gave impetus to apply BNP methods to now significantly more complex data structures and applications where $X_{1},\ldots,X_{n}$ are commonly viewed as latent variables. For instance in a Bayesian mixture model~\cite{Lo1984}, one may model observables 
$Y_{1},\ldots,Y_{n}|X_{1},\ldots,X_{n},P$ with some joint distribution  $\pi(Y_{1},\ldots,Y_{n}|\mathbf{X},P)$ and a Bayesian model for statistical inference is completed by working with  posterior distribution of $P|X_{1},\ldots,X_{n}$ and the marginal distribution of $X_{1},\ldots,X_{n}$ and perhaps priors on parametric components. Inference is facilitated by computational procedures that exploit the explicit nature and tractability of the DP posterior and marginal distributions. In particular, let $P\sim \mathrm{DP}(\theta,H)$ denote a Dirichlet process with parameter $\theta>0$ and (non-random) base measure $H$ on $\Xi,$ such that $\mathbb{E}[P]=H.$ We assume $H$ to be non-atomic but it can be defined more generally. The general joint marginal distribution of $X_{1},\ldots,X_{n}$ follows an exchangeable Blackwell-MacQueen P\'olya Urn distribution~\cite{BMQ,Pit96}. Due to the discrete nature of the DP, this distribution admits a random number, denoted as $K_{n}\leq n$, of unique i.i.d. values $(\tilde{X}_{1},\ldots,\tilde{X}_{k})$ for $K_{n}=k$, as well as ties. These can be expressed in terms of clusters $\{A_{1},\ldots,A_{k}\}$, where $A_{j}=\{i:X_{i}=\tilde{X}_{j}\}$ forms a random partition of $[n]=\{1,2,\ldots,n\}$. The distribution of the random partition $\{A_{1},\ldots,A_{k}\}$ follows a Chinese restaurant process, denoted as $\mathrm{CRP}(0,\theta)$, which only depends on the sizes $|A_{j}|=n_{j}$. It can be expressed, as described in~\cite{Pit96}, in terms of the exchangeable partition probability function (EPPF) $p_{0,\theta}$ on $\mathbb{N}=\{1,2,\ldots\}$, restricted to $[n],$ as follows,
$
p_{0,\theta}(n_{1},\ldots,n_{k})=\frac{\theta^{k}\Gamma(\theta)}{\Gamma(\theta+n)}\prod_{i=1}^{k} {\Gamma(n_{i})},
$
and hence the joint distribution of $(X_{1},\ldots,X_{n})$ may be expressed as 
\begin{equation}
p_{0,\theta}(n_{1},\ldots,n_{k})\times \prod_{j=1}^{k}H(d\tilde{X}_{j})=\frac{\theta^{k}\Gamma(\theta)}{\Gamma(\theta+n)}\prod_{i=1}^{k} {\Gamma(n_{i})}\times \prod_{j=1}^{k}H(d\tilde{X}_{j})
\label{jointDP}
\end{equation}
Equivalently, the joint distribution may be built sequentially using the prediction rule $X_{n+1}|X_{1},\ldots,X_{n}$ given by
\begin{equation}
\mathbb{E}[P|X_{1},\ldots,X_{n}]=\frac{\theta}{\theta+n}H+\frac{n}{\theta+n}\hat{P}_{n}
\label{predictDP}
\end{equation}
which also encodes the Chinese restaurant process seating procedure.
\cite[Theorem 1]{Ferg73} shows that the posterior distribution of the DP is again a DP with parameter $\theta+n$ and base measure $\frac{\theta}{\theta+n}H+\frac{n}{\theta+n}\hat{P}_{n}$, as in~\eqref{predictDP}.

As discussed in more detail in~\cite{Pit96,Pit03} all species sampling models, that is discrete random probability measures, of the form $P=\sum_{k=1}^{\infty}P_{k}\delta_{V_{k}}$ where $(P_{k})_{k\ge 1}:=(P_{k})$ is a sequence of probabilities summing to $1,$ independent of $(V_{k})\overset{iid}\sim H$ are associated with a unique EPPF which is in bijection to the distribution of $(P_{k}).$ While the DP, which is a species sampling model, remains the all-purpose tool for BNP applications, there has been considerable interest in other discrete priors for $P,$ both in terms of theory and practical applications, that may exhibit comparable properties to the DP. Moreover there is particular interest in classes of models that exhibit modeling capabilities that the DP does not.  The EPPF $p_{0,\theta}$ of the DP has the desirable property of having Gibbs (product) form consistent for each restriction of $[n].$ This facilitates computational procedures and also has theoretical implications. As established by~\cite{Gnedin06,Pit03,Pit06} the  classes of $P$ admitting EPPFs of Gibbs form, while large is actually quite exclusive. This coincides with the class of $\alpha$-Gibbs priors which are defined generally for $-\infty\leq \alpha\leq 1.$ In particular here we consider the case of $0\leq \alpha<1,$ which is essentially generated by conditioning on the total sum of the jumps of an $\alpha$-stable subordinator, and then mixing with respect to a probability measure. The most prominent example of this class is the now popular Pitman-Yor~\cite{IJ2001,Pit96} process indexed by parameters $0\leq \alpha<1$ and $\theta>-\alpha.$ The DP arises as a special case by taking $\alpha=0.$ Similar to the DP, tractable expressions for the posterior distribution and marginal distribution of the PY process has been obtained in~\cite{Pit96}. Another prominent closely related class is the class of normalized generalized gamma models. In general for $0<\alpha<1,$ $\alpha$-Gibbs priors exhibit various modeling capabilities for statistical inference, not exhibited by the DP. We cite a few references~\cite{Bacallado,LomeliFavaro,FavaroLijoi2009,LijoiSpecies} with more extensive descriptions and literature review in~\cite{Deblasi}.
Furthermore insights and interpretations of more choices of such priors, which play a subtle role in our results, and relations to operations arising in combinatorial stochastic processes, may be obtained from~\cite{HJL,HJL2,HJL3,PYaku},  

As mentioned, \cite{Deblasi} provide a comprehensive survey of statistical applications and attributes exploiting the unique properties  of the $\alpha$-Gibbs class. These include applications to mixture models and novel methods for species sampling models in relation to population genetics. The latter class of applications provide instances where the DP model is wholly inadequate. They provide an affirmation to their question "Are Gibbs-Type Priors the Most Natural
Generalization of the Dirichlet Process?." However~\cite[p.218]{Deblasi} also note the following important missing component, 
\begin{quote}
For our purposes it is enough to
note that the derivation of posterior quantities in this setting
represents a challenging issue, which has not found a
satisfactory solution to date 
\end{quote}
Hence 
while descriptions of the EPPFs, and corresponding exchangeable distributions for $(X_{1},\ldots,X_{n})$ are available and as demonstrated in ~\cite{Deblasi} much can be done with this, the description of the posterior distribution of $P|X_{1},\ldots,X_{n}$ in the general $0<\alpha<1$ range, that is to say the range exhibiting power law behavior, is unavailable. Similar to the parametric Bayesian setting, where much can be deduced from the marginal distribution of $(X_{1},\ldots,X_{n})$, a description of the posterior distribution of $P|X_{1},\ldots,X_{n}$ remains central to BNP analysis.

The purpose of this paper is to provide an answer to this important missing component. That is we obtain two equivalent descriptions of the posterior distribution of $P|X_{1},\ldots,X_{n}$ for all $0<\alpha<1$ Gibbs priors. The first description is of a form similar to the description of posterior distributions for general species sampling models $P,$ as described in~\cite{Ferg73,IJ2003,JLP,Pit96}.  While the second description is quite specific to $\alpha$-Gibbs priors. Furthermore our proof of the posterior distribution is of interest as it is specifically exploits the Gibbs form of the EPPF. It is notable that a perhaps more crude form of the posterior distribution, without understanding the appropriate refinements,  may be obtained by utilizing the methods in~\cite[Section 7]{James2002}.
Moreover our posterior results encode important operations related to stick-breaking representations and size-biased sampling as well as Markov Chains in relation to coagulation and fragmentation operations as described in~\cite{BerFrag,DGM,HJL2,Pit06}. We believe that some of the latter operations, as in the closely related operations used in~\cite{Gasthaus,Wood}, will begin to play a more prominent role in statistical applications. However, we shall not elaborate on this here. 
\subsection{Outline}
The exposition proceeds as follows. Section 2 provides details of the random probability measure formed by normalizing an $\alpha$- stable subordinaor, this also corresponds to a  Pitman-Yor process with parameters $(\alpha,0)$ which parallel the Dirichlet process. We then proceed to describe properties of the general $\alpha$ Gibbs priors, which are constructed from $(P_{k})$ following $\alpha$ stable Poisson Kingman distributions with mixing as described in~\cite{Pit03,Pit06}.  Importantly we describe its known EPPF, prediction rule and the analogous marginal distribution of $(X_{1},\ldots,X_{n}).$ Section 3 provides the main results describing the posterior distribution for the general Gibbs class, where known results for the normalized stable process corresponding to a PY process with parameters $(\alpha,0)$ plays a central role. We present a novel proof which exploits the special Gibbs properties of theses priors. 
Section 4 proceeds with  applications of our results to two specific cases for further illustration.
\section{Normalized Stable, Pitman-Yor and general Gibbs priors}
We start with defining the canonical Gibbs prior for all $0<\alpha<1$, which is the random probabilty measure defined by normalizing a stable subordinator. This is also a Pitman-Yor process with parameters $(\alpha,0)$ and is the bookend of the DP which is a Pitman-Yor process with parameters $(0,\theta)$ for some $\theta>0.$
Following~\cite{Kingman75,PPY92,Pit03,Pit06,PY97}
let $(T_{\alpha}(t): t\ge 0)$ denote an $\alpha$ stable subordinator, with $T_{\alpha}(0):=0,$ which means, as for the case of any subordinator, that for $s<t,$  $T_{\alpha}(t)-T_{\alpha}(s)\overset{d}=T_{\alpha}(t-s)$ and is independent of $T_{\alpha}(s).$ Furthermore, special to the stable case, for each $s,$ 
$T_{\alpha}(s)\overset{d}=s^{
{1}/{\alpha}}T_{\alpha}(1),$ where $T_{\alpha}(1):=T_{\alpha}$ is stable random variable with density $f_{\alpha}(t)$ and Laplace transform $\mathbb{E}[{\mbox e}^{-\lambda T_{\alpha}}]={\mbox e}^{-\lambda^{\alpha}}.$ Now considering $(T_{\alpha}(t): t\in[0,1]),$ we may construct a random distribution function for each $y\in[0,1]$ as, 
$$
F_{\alpha,0}(y)\overset{d}=\frac{T_{\alpha}(y)}{T_{\alpha}(1)}\overset{d}=\sum_{k=1}^{\infty}P_{k}\mathbb{I}_{\{U_{k}\leq y\}},
$$
where here $(P_{k})\sim \mathrm{PD}(\alpha,0)$ denoting it follows a Poisson-Dirichlet distribution with parameters $(\alpha,0),$ formed  by normalizing the ranked jumps of $(T_{\alpha}(t): t\in[0,1]),$ and is independent of $(U_{k})\overset{iid}\sim\mathrm{Uniform}[0,1].$
In turn, sampling $X_{1},\ldots,X_{n}|F_{\alpha,0}\overset{iid}\sim F_{\alpha,0},$
that is, $(X_{i}\overset{d}=F_{\alpha,0}^{-1}(U'_{i});$ $U'_{i}\overset{iid}\sim\mathrm{Uniform}[0,1], i\in [n])$, produces $K_{n}=k$ unique values $(\tilde{X}_{1},\ldots,\tilde{X}_{k})\overset{iid}\sim\mathrm{Uniform}[0,1]$ and a $\mathrm{CRP}(\alpha,0)$-partition of $[n],$ $\{A_{1},\ldots,A_{k}\},$ where $A_{j}=\{i:X_{i}=\tilde{X}_{j}\},$
with size $|A_{j}|=n_{j},$ for $j=1,\ldots,k,$ with an EPPF denoted by 
\begin{equation}
p_{\alpha}(n_{1},\ldots,n_{k})=\frac{\alpha^{k-1}\Gamma(k)}{\Gamma(n)}\prod_{j=1}^{k}(1-\alpha)_{n_{j}-1},
\label{canonEPPF}
\end{equation}
where, $(x)_n = x(x+1)\cdots(x+n-1) = {\Gamma(x+n)}/{\Gamma(x)}$ denotes the Pochhammer symbol, for any non-negative integer $x$. 
Furthermore, the probability of the number of blocks $K_{n}=k$ can be expressed as $\mathbb{P}_{\alpha,0}^{(n)}(k):=\mathbb{P}_{\alpha,0}(K_n = k)=
{\alpha^{k-1}\Gamma(k)} S_\alpha(n,k) / {\Gamma(n)}$, where $S_\alpha(n,k) =  
[{\alpha^k k!}]^{-1} \sum_{j=1}^k (-1)^j \binom{k}{j} (-j\alpha)_n$ denotes the generalized Stirling number of the second kind.

Recall that a DP may be defined by normalizing a gamma subordinator, and corresponds to a Pitman-Yor process with parameters $(0,\theta).$  
$F_{\alpha,0}$ is a the random cdf corresponding to a Pitman-Yor process with parameters $(\alpha,0)$ where its base measure is $\mathbb{U}$ corresponding to a $\mathrm{Uniform}[0,1]$ distribution.
Now notice for any (non-atomic) $H$ that $P_{\alpha,0}(\cdot):=F_{\alpha,0}(H(\cdot))$ only alters the values of the atoms of $F_{\alpha,0}$ and hence  $P_{\alpha,0}$ with law denoted $\mathscr{PY}(\alpha,0,H)$ is also a Pitman-Yor process with parameters $(\alpha,0)$ and base measure $H.$  
\begin{rem}
We shall use the notation $F\circ H$ to denote the composition of functions.
\end{rem}

Note that important to applications  involving species sampling as described in \cite{Deblasi,FavaroLijoi2009,LijoiSpecies}, unlike the DP case, the total mass $T_{\alpha}$ is not independent of $(P_{k}).$
Now we mention a few more technical matters which also has broader relevance to the origins of the study of $(P_{k})$ as described in~\cite{Pit03,Pit06,PY97}.  Letting $(L_{t}; t>0)$ denote the local time process, starting at $0,$ of a strong Markov process whose lengths of excursion on $[0,1]$ correspond to $(P_{k})\sim \mathrm{PD}(\alpha,0)$ and 
$T_{\alpha}(\ell)=\inf\{t:L_{t}>\ell\}, \ell\ge 0$, corresponds to its inverse local time. Due to the scaling identity
(see~\cite{PY92}),
\begin{equation}
L_{1}\overset{d}=\frac{L_{t}}{t^{\alpha}}\overset{d}=\frac{s}{[{{T_{\alpha}(s)]}^{\alpha}}}\overset{d}=T^{-\alpha}_{\alpha},
\label{scaling}
\end{equation}
the local time up to time $1,$ $L_{1}\overset{d}=T^{-\alpha}_{\alpha},$ follows a Mittag-Leffler distribution with density $g_{\alpha}(z):=f_{\alpha}(z^{-
1/\alpha})z^{-
{1}/\alpha-1}/\alpha,$ satisfying
\begin{equation}
L_{1}:=\Gamma(1-\alpha)^{-1}\lim_{\epsilon\rightarrow
0}\epsilon^{\alpha}|\{i:P_{i}\ge \epsilon\}|~\mathrm{a.s.}
\label{inverselocaltime}
\end{equation}
Relating to~\cite{Deblasi,FavaroLijoi2009,LijoiSpecies}, $T^{-\alpha}_{\alpha}$ corresponds to the $\alpha$-diversity of $(P_{k})\sim \mathrm{PD}(\alpha,0)$ and is the almost sure limit of $n^{-\alpha}K_{n},$ as shown in \cite{Pit03}.

\subsection{$\alpha$-stable Poisson-Kingman distributions and Gibbs partitions}
Now conditioning $(P_{i})$ on $T_{\alpha}=t$~(or $L_{1}=t^{-\alpha}$) leads to the distribution of $(P_{i})|T_{\alpha}=t\sim \mathrm{PD}(\alpha|t)$, and for $h(t)$ a non-negative function with $\mathbb{E}[h(T_{\alpha})]=1,$ one may, as in~\cite{Pit03}, define the $\alpha$-stable Poisson-Kingman distribution with mixing distribution $\nu(dt)/dt=h(t)f_{\alpha}(t),$ and write $(P_{\ell})\sim \mathrm{PK}_{\alpha}(h\cdot f_{\alpha})$ defined as 
$$
\mathrm{PK}_{\alpha}(h\cdot f_{\alpha}):=\int_{0}^{\infty}\mathrm{PD}(\alpha|t)h(t)f_{\alpha}(t)dt=\int_{0}^{\infty}\mathrm{PD}(\alpha|s^{-\frac{1}{\alpha}})h(s^{-\frac{1}{\alpha}})g_{\alpha}(s)ds.
$$
Relevant to the construction of general Pitman-Yor processes with parameters $(\alpha,\theta),$ setting $h(t)=t^{-\theta}/\mathbb{E}[T^{-\theta}_{\alpha}]$, for $\theta>-\alpha$, leads to $(P_{\ell})\sim \mathrm{PD}(\alpha,\theta),$ corresponding to the important two-parameter Poisson-Dirichlet distribution as described in~\cite{PPY92, Pit96,Pit03,Pit06,PY97}, whose size biased re-arrangement, say, $(\tilde{P}_{\ell})\sim \mathrm{GEM}(\alpha,\theta),$ where $\mathrm{GEM}(\alpha,\theta)$ is the two-parameter Griffiths-Engen-McCloskey distribution, and is widely used in applications~\cite{BerFrag,IJ2001,PPY92,Pit96,Pit06}. 
The inverse local time at $1$ of a process with lengths $(P_{\ell})\sim \mathrm{PD}(\alpha,\theta),$ say, $T_{\alpha,\theta},$ has density $f_{\alpha,\theta}(t)=t^{-\theta}f_{\alpha}(t)/\mathbb{E}[T^{-\theta}_{\alpha}]$, and the corresponding local time at $1$ or its $\alpha$-diversity, $T^{-\alpha}_{\alpha,\theta}\sim \mathrm{ML}(\alpha,\theta),$ denoting that it has a
 $(\alpha,\theta)$ generalized Mittag-Leffler distribution with density $g_{\alpha,\theta}(s)=s^{\theta/\alpha}g_{\alpha}(s)/\mathbb{E}[T^{-\theta}_{\alpha}]$. This choice of $(P_{k})\sim \mathrm{PD}(\alpha,\theta),$  leads to $P_{\alpha,\theta}\sim \mathscr{PY}(\alpha,\theta,H),$ corresponding to a general Pitman-Yor process. Another important case is the normalized generalized gamma process which corresponds to $h(t)={\mbox e}^{-\lambda t}{\mbox e}^{\lambda^{\alpha}}$ for some $\lambda>0.$
In general when $(P_{\ell})\sim \mathrm{PK}_{\alpha}(h\cdot f_{\alpha})$, $F(y):=\sum_{k=1}^{\infty}P_{k}\mathbb{I}_{\{U_{k}\leq y\}}$ is the corresponding random cdf defined similarly as $F_{\alpha,0}$, and sampling $n$ variables from $F$, as in the $\mathrm{PD}(\alpha,0)$ case, leads to the general class of $\alpha\in(0,1)$ Gibbs partitions of $[n]$, with an EPPF , denoted by $\mathrm{PK}_{\alpha}(h\cdot f_{\alpha})-\mathrm{EPPF},$ as in ~\cite{Gnedin06,Pit03,Pit06},
\begin{equation}
p^{[\nu]}_{\alpha}(n_{1},\ldots,n_{k})={\Psi}^{[\alpha]}_{n,k}\times p_{\alpha}(n_{1},\ldots,n_{k}),
\label{VEPPF}
\end{equation}
where, important for our exposition, using the interpretation of the expressions derived by \cite{Gnedin06,Pit03,Pit06} in~\cite{HJL,HJL2},
$$
{\Psi}^{[\alpha]}_{n,k} =\mathbb{E}[h(T_{\alpha})|K_{n}=k]=\mathbb{E}\bigg[h\big(Y^{(n-k\alpha)}_{\alpha,k\alpha}\big)\bigg].
$$ 
In the first expectation, $T_{\alpha}|K_{n}=k$ is evaluated for the $\mathrm{PD}(\alpha,0)$ case with $K_{n}\sim 
\mathbb{P}_{\alpha,0}^{(n)}(k).$ The second equality follows from the fact that $T_{\alpha}|K_{n}=k$ equates in distribution 
to a variable 
$Y^{(n-k\alpha)}_{\alpha,k\alpha},$ with density $f^{(n-k\alpha)}_{\alpha,k\alpha}(t),$ such that~(pointwise), as in~\cite[eq. (2.13), p. 323]{HJL2},
\begin{equation}
Y^{(n-k\alpha)}_{\alpha,k \alpha}\overset{d}=\frac{T_{\alpha,k\alpha}}{B_{k\alpha,n-k\alpha}}=\frac{T_{\alpha,n}}{B^{\frac1\alpha}_{k,\frac{n}{\alpha}-k}},
\label{jamesidspecial}
\end{equation}
where variables in each ratio are independent, and throughout, $B_{a,b}$ denotes a $\mathrm{Beta}(a,b)$ random variable. The expression in~\eqref{jamesidspecial} also correspond to the notion of a conditional $\alpha$-diversity as described in~\cite{Deblasi,FavaroLijoi2009,LijoiSpecies}. See the forthcoming Section~\ref{sec:interpret} for more  comments in relation to their work.

These points show that  $F$ is the random cdf corresponding to a Gibbs-type prior. Using $P:=F\circ H,$ we may write $P\sim \mathrm{PK}_{\alpha}(h\cdot f_{\alpha},H)$ to denote the general Gibbs-type prior with base measure $H,$ as the full generalization of $P_{\alpha,0}.$ With the above information we can now describe the marginal distribution of $X_{1},\ldots,X_{n}$ which, given $P,$ is sampled iid from $P.$ That is for $X_{1},\ldots,X_{n}|P$, iid $P$ and $P\sim\mathrm{PK}_{\alpha}(h\cdot f_{\alpha},H)$, it follows that the marginal distribution of  $X_{1},\ldots,X_{n}$ can be expressed as
\begin{equation}
{\Psi}^{[\alpha]}_{n,k}\times p_{\alpha}(n_{1},\ldots,n_{k})\times \prod_{j=1}^{k}H(d\tilde{X}_{j})
\label{marginal}
\end{equation}
where $(\tilde{X}_{1},\ldots,\tilde{X}_{k})$ are the $K_{n}=k,$ unique values of $X_{1},\ldots,X_{n}$ with respective multiplicities $(n_{1},\ldots,n_{k}).$ For clarity the distribution for $K_{n}$ in this instance is
\begin{equation}
\label{genKndist}
\mathbb{P}^{[\nu]}(K_{n}=k)={\Psi}^{[\alpha]}_{n,k}\mathbb{P}^{(n)}_{\alpha,0}(k)
\end{equation}
Equivalently one may also express the~\eqref{marginal} in terms of the product over conditional distributions of the form $X_{i}|X_{1},\ldots,X_{i-1}$. In particular the distribution of $X_{n+1}|X_{1},\ldots,X_{n}$ can be obtained from~\eqref{marginal} as 
\begin{equation}\label{eq:PKprediction_rule}
 \frac{{\Psi}^{[\alpha]}_{n+1,k+1}}{{\Psi}^{[\alpha]}_{n,k}}\frac{k\alpha}{n}H + \frac{{\Psi}^{[\alpha]}_{n+1,k}}{{\Psi}^{[\alpha]}_{n,k}}\sum_{j=1}^{k}\frac{n_{j}-\alpha}{n}\delta_{\tilde{X}_{j}}
\end{equation} 
In a Chinese restaurant metaphor this describes the chance of customer ${n+1}$ joining a new or existing table and choosing an existing value or new value from $H.$ It follows that the prediction rule is also equivalent to $\mathbb{E}[P\mid X_{1},\ldots,X_{n}],$
which naturally relates to the posterior, or conditional, distribution of $P\mid X_{1},\ldots,X_{n}.$
\begin{rem}Note the expression in~\eqref{eq:PKprediction_rule}, which is based on the results in~\cite{HJL,HJL2}, is equivalent to the more typical form of the prediction rule appearing in the literature which uses the form of the EPPF  $p^{[\nu]}_{\alpha}(n_{1},\ldots,n_{k})= V_{n,k}\prod_{j=1}^{k}(1-\alpha)_{n_{j}-1}$ for some weights satisfying $V_{n,k}=(n-k\alpha)V_{n+1,k}+V_{n+1,k+1}$ for $n=1,2,\ldots,$ and $k=1,\ldots,n$ with $V_{1,1}=1.$ See \cite[Section 1.2]{HJL2} for more details.
\end{rem}

The posterior distribution is fundamental for applications in Bayesian non-parametric statistics, and as we shall show, has various interpretations. However, with the exception of descriptions of the posterior distribution of the Pitman-Yor family $\mathscr{PY}(\alpha,\theta,H),$ obtained by~\cite[Corollary 20]{Pit96}, and the generalized gamma case with $h(t)={\mbox e}^{-\lambda t}{\mbox e}^{\lambda^{\alpha}},$ which can be obtained from~\cite{JLP}, descriptions of the posterior distributions for other choices of $h(t)$ are generally unknown. Our interest in this work is to provide two such descriptions for all $h$, and also comment briefly about connections to operations arising in~\cite{DGM,PPY92,PY97}.
\section{Posterior distributions of Gibbs-type priors}We now proceed to describe the posterior distribution of $P|X_{1},\ldots,X_{n}$ when
$P\sim \mathrm{PK}_{\alpha}(h\cdot f_{\alpha},H)$. We first recall the description of the posterior distribution of the $P=P_{\alpha,0}\sim \mathscr{PY}(\alpha,0,H),$ which can be read from \cite[Corollary 20]{Pit06} and is equal in distribution to

\begin{equation}\label{fpost3}
\tilde{P}^{(n)}_{\alpha,0}\overset{d}=B_{k\alpha,n-k\alpha}P_{\alpha,k\alpha} + (1 -B_{k\alpha,n-k\alpha}) \sum_{j=1}^k D_j \delta_{\tilde{X}_j},
\end{equation}
where $(D_{1},\ldots,D_{k})\sim \mathrm{Dir}(n_{1}-\alpha,\ldots,n_{k}-\alpha),$ denoting a Dirichlet vector with $k$ parameters as indicated, is independent of the also mutually independent variables $B_{k\alpha,n-k\alpha}\sim \mathrm{Beta}(k\alpha,n-k\alpha)$ and $P_{\alpha,k\alpha}\sim \mathscr{PY}(\alpha,k\alpha,H),$ with inverse local time $T_{\alpha,k\alpha},$ and otherwise $(\tilde{X}_{1},\ldots,\tilde{X}_{k})$ are the observed (hence fixed) unique values drawn iid from $H.$ A second, less well known description for the posterior distribution may be read from~\cite[Section 6.4, Proposition 6.6]{JamesPGarxiv}, and otherwise follows by applying~\cite[Proposition 21]{PY97} to~
\eqref{fpost3} to obtain
\begin{equation}\label{fpost4}
\tilde{P}^{(n)}_{\alpha,0}\overset{d}=P_{\alpha,n}\circ(B_{k,\frac{n}{\alpha}-k}H+(1-B_{k,\frac{n}{\alpha}-k})\sum_{j=1}^kd_j\delta_{\tilde{X}_j}),
\end{equation}
where $P_{\alpha,n}\sim \mathscr{PY}(\alpha,n,\mathbb{U})$, with inverse local time $T_{\alpha,n},$
independent of the random probability measure $B_{k,\frac{n}{\alpha}-k}H+(1-B_{k,\frac{n}{\alpha}-k})\sum_{j=1}^kd_j\delta_{\tilde{X}_j},$ where furthermore 
$(d_{1},\ldots,d_{k})\sim \mathrm{Dir}(\frac{n_{1}-\alpha}{\alpha},\ldots,\frac{n_{k}-\alpha}{\alpha}),$ is independent of  $B_{k,\frac{n}{\alpha}-k}.$

\begin{rem}
One may notice that the representations in~\eqref{fpost3} and \eqref{fpost4} correspond to the respective representations of the distribution of $T_{\alpha}|K_{n}=k$ in~\eqref{jamesidspecial}.
\end{rem}

\begin{rem}
A significant feature of the representation in~\eqref{fpost4} is that it encodes the dual coagulation and $\mathrm{PD}(\alpha,1-\alpha)$ fragmentation operations described in~\cite{DGM}, by setting $n=1.$ The case of general $n$ extends this notion. 
\end{rem}

Define, for some fixed positive $\omega>0,$ 
\begin{equation}
\label{hpost}
h^{(\omega)}_{\alpha,\theta}(t)=\frac{t^{-\theta}h(t/\omega)}{\mathbb{E}[h(T_{\alpha,\theta}/\omega)]\mathbb{E}[T^{-\theta}_{\alpha}]}.                                                                                                                            
\end{equation}
We shall encounter $\mathrm{PK}_{\alpha}(h^{(\omega)}_{\alpha,\theta}\cdot f_{\alpha})$ distributions in the results below in particular for the choices of $\theta=k\alpha$ and $\omega=b$ leading to $h^{(b)}_{\alpha,k\alpha}(t)$ and $\theta=n,\omega=b^{1/\alpha}$ leading to $h^{(b^{1/\alpha})}_{\alpha,n}(t).$ With these results, we proceed with two descriptions of the posterior distribution of $P|X_{1},\ldots,X_{n}$ for $P\sim\mathrm{PK}_{\alpha}(h\cdot f_{\alpha},H).$

\subsection{A first description of the posterior distribution for 
$P\sim\mathrm{PK}_{\alpha}(h\cdot f_{\alpha},H)$}

In the first description we shall encounter inverse local times
$\hat{T}_{\alpha,k\alpha}$ with marginal density $\hat{h}_{\alpha,k\alpha}(t)f_{\alpha}(t),$ where for $R_{k}\sim \mathrm{Beta}(k\alpha,n-k\alpha),$
\begin{equation}
\label{hpost1}
\hat{h}_{\alpha,k\alpha}(t)=\frac{t^{-k\alpha}\mathbb{E}[h(t/R_{k})]}{\mathbb{E}[h(T_{\alpha,k\alpha}/R_{k})]\mathbb{E}[T^{-k\alpha}_{\alpha}]}
=\frac{t^{-k\alpha}\mathbb{E}[h(t/B_{k\alpha,n-k\alpha})]}
{\mathbb{E}[T^{-k\alpha}_{\alpha}]{\Psi}^{[\alpha]}_{n,k}}.                                                                                                                            
\end{equation}

\begin{thm}\label{thm:post1}
Let $P\sim\mathrm{PK}_{\alpha}(h\cdot f_{\alpha},H)$ with inverse local time $T$ having density $h(t)f_{\alpha}(t)dt$. Then, given $X_1, \ldots, X_n$, the posterior distribution of $P$ is equivalent to the distribution of the following random probability measure:
\begin{equation}\label{fpost1}
\tilde{P}^{(n)}\overset{d}=\hat{R}_k \hat{P}_{\alpha,k\alpha} + (1 -\hat{R}_k) \sum_{j=1}^{k} D_j \delta_{\tilde{X}_{j}},
\end{equation}
where $\hat{P}_{\alpha,k\alpha}$ has $\alpha$ diversity $\hat{T}^{-\alpha}_{\alpha,k\alpha}.$ The variables satisfy the following properties 
\begin{enumerate}
\item $(D_1, \ldots, D_k) \sim \mathrm{Dir}(n_1 - \alpha, \ldots, n_k - \alpha)$ independent of  $(\hat{R}_k, \hat{T}_{\alpha,k\alpha})$
\item The posterior distribution of $T|K_{n}=k$ is equivalent to the distribution of $\hat{T}_{\alpha,k\alpha}/\hat{R}_{k}$.
\item The joint density of $(\hat{R}_{k},\hat{T}_{\alpha,k})$ is
\begin{equation}
f_{\hat{R}_k,\hat{T}_{\alpha,k\alpha}}(b,t) = \frac{[h(t/b)]}{\mathbb{E}[h(T_{\alpha,k\alpha}/R_k)]}f_{R_k}(b)f_{\alpha,k\alpha}(t),
\label{jointRT}
\end{equation}
where $\mathrm{R_k} \sim \mathrm{Beta}(k\alpha, n-k\alpha)$, with density $f_{R_{k}}.$

\item $\hat{R}_k$ has marginal density
$$
f_{\hat{R}_k}(b) = \frac{\mathbb{E}[h(T_{\alpha,k\alpha}/b)]}{{\Psi}^{[\alpha]}_{n,k}}f_{R_k}(b),
$$

\item $\hat{P}_{\alpha,k\alpha}|\hat{R}_{k}=b\sim \mathrm{PK}_{\alpha}(h^{(b)}_{\alpha,k\alpha}\cdot f_{\alpha},H)$ meaning that $\hat{P}_{\alpha,k\alpha}\overset{d}=\sum_{l=1}^{\infty}W_{l}\delta_{V'_{l}}$, where $(W_{v})|\hat{R}_{k}=b$ has the $\mathrm{PK}_{\alpha}$ law specified by
$$
\int_{0}^{\infty} \mathrm{PD}(\alpha | t) \frac{h(t/b)}{\mathbb{E}[h(T_{\alpha,k\alpha}/b)]} f_{\alpha,k\alpha}(t) dt:=
\mathrm{PK}_{\alpha}(h^{(b)}_{\alpha,k\alpha}\cdot f_{\alpha})
$$
where $f_{\alpha,k\alpha}(t)$ is the density of $T_{\alpha,k\alpha},$
independent of $(V'_{l})\overset{iid}\sim H.$
\item The marginal density of $\hat{T}_{\alpha,k\alpha}$ can be expressed as $\hat{h}_{\alpha,k\alpha}(t)f_{\alpha}(t),$ specified as in~\eqref{hpost1} which means that 
$
(W_{l})\sim \mathrm{PK}_{\alpha}(\hat{h}_{\alpha,k\alpha}\cdot f_{\alpha}).
$
\end{enumerate}
\end{thm}

\subsection{A second description of the posterior distribution for 
$P\sim\mathrm{PK}_{\alpha}(h\cdot f_{\alpha},H)$}\label{sec:postfrag}
We now proceed to provide another description of the posterior distribution based on \eqref{fpost4}. Similar to \eqref{hpost1} , there is the inverse local time $\hat{T}_{\alpha,n}$ with marginal density  $\tilde{h}_{\alpha,n}(t)f_{\alpha}(t),$ where for $\beta_{k}\sim \mathrm{Beta}(k,\frac{n}{\alpha}-k),$
\begin{equation}
\label{hpost2}
\tilde{h}_{\alpha,n}(t)=\frac{t^{-n}\mathbb{E}[h(t/\beta^{1/\alpha}_{k})]}{\mathbb{E}[h(T_{\alpha,n}/\beta^{1/\alpha}_{k})]\mathbb{E}[T^{-n}_{\alpha}]}=\frac{t^{-n}\mathbb{E}[h(tB^{-1/\alpha}_{k,\frac{n}{\alpha}-k})]}{{\Psi}^{[\alpha]}_{n,k}\mathbb{E}[T^{-n}_{\alpha}]}
\end{equation}
\begin{rem}
As we shall mention again in Section~\ref{sec:interpret}, the density in~\eqref{hpost2} equates with density in \cite[eq. (3.8)]{HJL2}, in relation to applying sequentially the  $\mathrm{PD}(\alpha,1-\alpha)$ fragmentation operation of ~\cite{DGM} $n$ times to $(P_{k})\sim \mathrm{PK}_{\alpha}(h\cdot f_{\alpha}).$
\end{rem}

\begin{thm}\label{thm:post2}
Let $P \sim \mathrm{PK}_{\alpha}(h\cdot f_{\alpha},H)$. Then, given $X_1, \ldots, X_n$, the posterior distribution of $P$ is equal to the distribution of the following random probability measure:
\begin{equation}
\label{fpost2}
\tilde{P}^{(n)}\overset{d}=\hat{P}_{\alpha,n}\circ (\hat{\beta}_{k}H+(1-\hat{\beta}_{k})\sum_{j=1}^{k}d_j\delta_{\tilde{X}_{j}})
\end{equation}
\begin{enumerate}
\item $(d_1, \ldots, d_k) \sim \mathrm{Dir}(\frac{n_1 - \alpha}{\alpha}, \ldots, \frac{n_k - \alpha}{\alpha})$, independent of $(\hat{\beta}_k, \hat{P}_{\alpha,n})$
\item The distribution of $T|K_{n}=k$ is equivalent to $\hat{T}_{\alpha,n}/\hat{\beta}^{1/\alpha}_{k},$ where $(\hat{\beta}_{k},\hat{T}_{\alpha,n})$ have joint density,
\begin{equation}
\label{genjointfragden}
f_{\hat{\beta}_{k},\hat{T}_{\alpha,n}}(b,t)=\frac{h(t/b^{1/\alpha})}{\mathbb{E}[h(T_{\alpha,n}/\beta^{1/\alpha}_k)]}f_{{\beta}_k}(b)f_{\alpha,n}(t),
\end{equation}
\item $\hat{\beta}_k$ has marginal density
$$
f_{\hat{\beta}_{k}}(b) = \frac{\mathbb{E}[h(T_{\alpha,n}/b^{1/\alpha})]}{\mathbb{E}[h(T_{\alpha,n}/\beta^{1/\alpha}_k)]}f_{{\beta}_k}(b),
$$
where $\beta_k \sim \mathrm{Beta}(k, n/\alpha-k)$, with density $f_{\beta_{k}}.$
\item $\hat{P}_{\alpha,n}|\hat{\beta}_{k}=b$ is such that $\hat{P}_{\alpha,n}\overset{d}=\sum_{v=1}^{\infty}Q_{v}\delta_{\tilde{U}_{v}}$, where $(Q_{v})|\hat{\beta}_{k}=b$ has the $\mathrm{PK}_{\alpha}$ law specified by
$$
\int_{0}^{\infty} \mathrm{PD}(\alpha | t) \frac{h(t/b^{1/\alpha})}{\mathbb{E}[h(T_{\alpha,n}/b^{1/\alpha})]} f_{\alpha,n}(t) dt:=\mathrm{PK}_{\alpha}(h^{(b^{1/\alpha})}_{\alpha,n}\cdot f_{\alpha}),
$$
where $f_{\alpha,n}(t)$ is the density of $T_{\alpha,n},$ independent of $(\tilde{U}_{v})\overset{iid}\sim \mathbb{U}.$
\item The marginal density of $\hat{T}_{\alpha,n}$ can be expressed as $\tilde{h}_{\alpha,n}(t)f_{\alpha}(t),$ specified as in~\eqref{hpost2} which means that 
$$
(Q_{v})\sim \mathrm{PK}_{\alpha}(\tilde{h}_{\alpha,n}\cdot f_{\alpha}):=\int_{0}^{\infty}\mathrm{PD}(\alpha|t)
\frac{\mathbb{E}[h(tB^{-1/\alpha}_{k,\frac{n}{\alpha}-k})]}{{\Psi}^{[\alpha]}_{n,k}}f_{\alpha,n}(t)dt.
$$
with distribution also appearing in \cite[eq. (3.8) and Proposition 3.3]{HJL2}.
\end{enumerate}
\end{thm}

\subsection{Proof of Theorems \ref{thm:post1} and \ref{thm:post2}}

We now proceed with the proof of both Theorems \ref{thm:post1} and \ref{thm:post2}.
\begin{proof}
We use the fact that for $X_{1},\ldots,X_{n}|P$ iid $P$ for $P\sim \mathrm{PK}_{\alpha}(h\cdot f_{\alpha},H),$ the marginal distribution of $X_{1},\ldots,X_{n}$ is equivalent to 
\begin{equation}
{\Psi}^{[\alpha]}_{n,k}\times p_{\alpha}(n_{1},\ldots,n_{k})\prod_{j=1}^{k}H(d\tilde{X}_{j}),
\label{marg}
\end{equation}
where $(\tilde{X}_{1},\ldots,\tilde{X}_{k})$ are the $k$ unique values of $X_{1},\ldots,X_{n}$, and $\mathbb{E}[h(T_{\alpha})|K_{n}=k]={\Psi}^{[\alpha]}_{n,k}$.

The posterior distribution of $P|X_{1},\ldots,X_{n}$ is determined by $\mathbb{E}[\Omega(P)|X_{1},\ldots,X_{n}]$ in the double expectation,
\begin{equation}
\mathbb{E}[\Omega(P)]= \mathbb{E}^{\nu}_{\alpha,n}\left(\mathbb{E}[\Omega(P)|X_{1},\ldots,X_{n}]\right)
\label{expect1} 
\end{equation}
where~\eqref{expect1} holds for arbitrary, hence all,  measurable functions $\Omega$,
where $\mathbb{E}^{\nu}_{\alpha,n}$ denotes the expectation with respect to the marginal distribution of $X_{1},\ldots,X_{n}$ as described in~\eqref{marg}.
Now notice there is the equivalence
$$
\mathbb{E}[\Omega(P)]=\mathbb{E}[\Omega(P_{\alpha,0})h(T_{\alpha})].
$$ 
which is equivalent to
\begin{equation}
\mathbb{E}_{\alpha,n}\left(\mathbb{E}[\Omega(P_{\alpha,0})h(T_{\alpha})|X_{1},\ldots,X_{n}]\right)
\label{expect2}
\end{equation}
where $\mathbb{E}_{\alpha,n}$ denotes expectation with respect to the marginal distribution of $X_{1},\ldots,X_{n}|P_{\alpha,0}\overset{iid}\sim P_{\alpha,0}$ for $P_{\alpha,0}\sim\mathscr{PY}(\alpha,0,H),$ which for clarity is 
$$p_{\alpha}(n_{1},\ldots,n_{k})\prod_{j=1}^{k}H(d\tilde{X}_{j}).$$
Furthermore, the inner expectation $\mathbb{E}[\Omega(P_{\alpha,0})h(T_{\alpha})|X_{1},\ldots,X_{n}]$ denotes expectation with respect to the posterior distribution of $P_{\alpha,0}$ and $T_{\alpha}$ given  $X_{1},\ldots,X_{n}$ as described in~\eqref{fpost3} and~\eqref{fpost4}.
That is,
\[
\begin{aligned}
\mathbb{E}[\Omega(P_{\alpha,0})h(T_{\alpha})|X_{1},\ldots,X_{n}] &= \mathbb{E}[\Omega(\tilde{P}^{(n)}_{\alpha,0})h(T_{\alpha,k\alpha}/B_{k,n-k\alpha})] \\
&= \mathbb{E}[\Omega(\tilde{P}^{(n)}_{\alpha,0})h(T_{\alpha,n}/B^{1/\alpha}_{k,\frac{n}{\alpha}-k})].
\end{aligned}
\]
where we use the conditional distribution of $T_{\alpha}|K_{n}=k$ as defined in~\eqref{jamesidspecial},
and $\tilde{P}^{(n)}_{\alpha,0}$ corresponds more precisely to the representation on the right hand-side of~\eqref{fpost3} and~\eqref{fpost4}, where all the respective variables are defined on the same space. Set $R_{k}=B_{k\alpha,n-k\alpha}$ and $\beta_{k}=B_{k,\frac{n}{\alpha}-k}.$
Now the Gibbs form of the EPPFs enables one to do a simple algebraic manipulation as follows
$$
\mathbb{E}[\Omega(P_{\alpha,0})h(T_{\alpha})|X_{1},\ldots,X_{n}]p_{\alpha}(n_{1},\ldots,n_{k})
$$
is equal to
$$
\frac{\mathbb{E}[\Omega(P_{\alpha,0})h(T_{\alpha})|X_{1},\ldots,X_{n}]}{\mathbb{E}[h(T_{\alpha})|K_{n}=k]}\times {\Psi}^{[\alpha]}_{n,k}\times p_{\alpha}(n_{1},\ldots,n_{k})
$$
Now using this combined with~\eqref{expect1} and~\eqref{expect2}, it follows that 
$$
\mathbb{E}[\Omega(P)|X_{1},\ldots,X_{n}]=\frac{\mathbb{E}[\Omega(P_{\alpha,0})h(T_{\alpha})|X_{1},\ldots,X_{n}]}{\mathbb{E}[h(T_{\alpha})|K_{n}=k]}
$$
completing the result.
\end{proof}
\subsection{Recovering the prediction rule in \eqref{eq:PKprediction_rule}}
As we noted, the prediction rule, that is the distribution of $X_{n+1}|X_{1},\ldots X_{n},$ can be expressed as in~\eqref{eq:PKprediction_rule} which we rewrite here for clarity
$$
\mathbb{E}[P | X_{1},\ldots,X_{n}] = \frac{{\Psi}^{[\alpha]}_{n+1,k+1}}{{\Psi}^{[\alpha]}_{n,k}}\frac{k\alpha}{n}H + \frac{{\Psi}^{[\alpha]}_{n+1,k}}{{\Psi}^{[\alpha]}_{n,k}}\sum_{j=1}^{k}\frac{n_{j}-\alpha}{n}\delta_{\tilde{X}_{j}}.
$$
However, it is not immediately clear that this equates with
$$
\mathbb{E}[\tilde{P}^{(n)}]=\mathbb{E}[\hat{R}_{k}\hat{P}_{\alpha,k\alpha}]+
\mathbb{E}[(1-\hat{R}_{k})]\sum_{j=1}^{k}\frac{n_{j}-\alpha}{n-k\alpha}\delta_{\tilde{X}_{j}}
$$
which is obtained from \eqref{fpost1}. We now provide some details for clarity. Since $\hat{P}_{\alpha,k\alpha}|\hat{R}_{k}=b$ has a Poisson-Kingman distribution hence corresponds to a species sampling model, it 
follows that $\mathbb{E}[\hat{P}_{\alpha,k\alpha}|\hat{R}_{k}=b]=H.$ Now notice that
$$
{\Psi}^{[\alpha]}_{n,k}\mathbb{E}[\hat{R}_{k}]=\int_{0}^{1}b\mathbb{E}[h(T_{\alpha,k\alpha}/b)]f_{R_{k}}(b)db
$$
where $R_{k}\sim \mathrm{Beta}(k\alpha,n-k\alpha).$ Hence 
$$
bf_{R_{k}}(b)=\frac{k\alpha}{n}f_{B_{k\alpha+1,n-k\alpha}}(b)
$$
where $f_{B_{k\alpha+1,n-k\alpha}}(b)$ is the density of $B_{k\alpha+1,n-k\alpha}\sim \mathrm{Beta}(k\alpha+1,n-k\alpha).$ Using this,
$$
{\Psi}^{[\alpha]}_{n,k}\mathbb{E}[\hat{R}_{k}]=
\frac{k\alpha}{n}\mathbb{E}[h(T_{\alpha,k\alpha}/B_{k\alpha+1,n-k\alpha})]
$$ 
Now for general $\theta>-\alpha,$ there is the known identity, see \cite{PPY92,Pit06},
$$
T_{\alpha,\theta}\overset{d}=\frac{T_{\alpha,\theta+\alpha}}{B_{\theta+\alpha,1-\alpha}}
$$
Setting $\theta=k\alpha,$ and using $B_{(k+1)\alpha,1-\alpha}B_{k\alpha+1,n-k\alpha}\overset{d}=B_{(k+1)\alpha,n+1-(k+1)\alpha},$ leads to,
$$
\frac{T_{\alpha,k\alpha}}{B_{k\alpha+1,n-k\alpha}}\overset{d}=\frac{T_{\alpha,(k+1)\alpha}}{B_{(k+1)\alpha,n+1-(k+1)\alpha}}.
$$
Hence $\mathbb{E}[h(T_{\alpha,k\alpha}/B_{k\alpha+1,n-k\alpha})]=
\mathbb{E}[h(T_{\alpha,(k+1)\alpha}/B_{(k+1)\alpha,n+1-(k+1)\alpha})]=
{\Psi}^{[\alpha]}_{n+1,k+1}.$ For $\mathbb{E}[(1-\hat{R}_{k})],$ it follows directly that,
$$ 
{\Psi}^{[\alpha]}_{n,k}\mathbb{E}[(1-\hat{R}_{k})]=\frac{n-k\alpha}{n}\mathbb{E}[h(T_{\alpha,k\alpha}/B_{k\alpha,n+1-k\alpha})]=\frac{n-k\alpha}{n}{\Psi}^{[\alpha]}_{n+1,k},
$$
by $n(1-b)f_{R_{k}}(b)=(n-k\alpha)f_{B_{k\alpha,n+1-k\alpha}}(b).$
\subsection{Interpretations of posterior components- species variety, stickbreaking and random fragmentation}\label{sec:interpret}
Posterior distribution representations encode valuable information about $P$ leading also to interesting operations. Here we sketch out some interesting results and interpretations relative to our result for $P\sim \mathrm{PK}_{\alpha}(h\cdot f_{\alpha},H).$ Following~\cite{Pit96} in a species sampling context, and certainly relevant to some of the work discussed in~\cite{Deblasi}, where $(X_{1},\ldots,X_{n})$ represent the possible values of the first $n$ animals observed in an eco-system given $K_{n}=k,$ denoting there are $k$ distinct species out of $n.$ One may represent $P=\sum_{l=1}^{\infty}P_{l}\delta_{V_{l}}$ in an almost sure fashion as $P=\sum_{j=k+1}^{\infty}P'_{j}\delta_{V'_{j}}+\sum_{j=1}^{k}\tilde{P}_{j}\delta_{\tilde{X}_{j}},$ where $(\tilde{P}_{1},\ldots,\tilde{P}_{k})\overset{d}=(1 -\hat{R}_k)(D_{1},\ldots,D_{k})$  represent the frequencies in order of appearance of the first $k$ species out of $n,$ according to $(X_{1},\ldots,X_{n}),$ that is according to a process of size-biased sampling with replacement, and 
$(P'_{j}:j\ge k+1):=(P_{i})\setminus(\tilde{P}_{1},\ldots,\tilde{P}_{k}),$ in no particular order, represent the frequencies of species types yet to be selected. This ties in with the study of new species in an additional sample of size $m$ initiated in the work of~\cite{LijoiSpecies} where $P$ follows a general Pitman-Yor process
with followup in the normalized generalized gamma case in~\cite{FavaroLijoi2009} as discussed in~\cite[Section 4.1]{Deblasi}.
That is  sampling $X_{n+1},\ldots, X_{n+m}$ additional samples given $X_{1},\ldots,X_{n},$ \cite{FavaroLijoi2009,LijoiSpecies} propose and examine the properties of $K^{(n)}_{m}$  the number of new species in the additional sample. For general $\alpha$-Gibbs priors we note that these new species come from sampling  $\sum_{j=k+1}^{\infty}P'_{j}\delta_{V'_{j}}$  which by our results may be set to $\hat{R}_{k}\hat{P}_{\alpha,k\alpha}.$ Since given $K_{n}=k,$ $\hat{P}_{\alpha,k\alpha}$ is a stable Poisson Kingman process with mixing density corresponding to the distribution of $\hat{T}_{\alpha,k\alpha},$ it follows immediately from \cite[Proposition 13]{Pit03} and scaling properties in~\eqref{scaling}, that as $m\rightarrow\infty,$ for fixed $n,$
$$
m^{-\alpha}K^{(n)}_{m}\overset{a.s}\rightarrow \hat{R}^{\alpha}_{k}\hat{T}^{-\alpha}_{\alpha,k\alpha}=\hat{\beta}_{k}\hat{T}^{-\alpha}_{\alpha,n}
$$
which is just the natural extension of the corresponding results in
\cite{FavaroLijoi2009,LijoiSpecies}, subject to a change of measure. Now setting $n=1$ in~\eqref{fpost1} and randomizing $X_{1}=\tilde{X}_{1}\sim H$ leads to the almost sure representation

\begin{equation}\label{stick}
P=\hat{R}_1\hat{P}_{\alpha,\alpha} + (1 -\hat{R}_1) \delta_{\tilde{X}_1},
\end{equation}
where one may set $\tilde{P}_{1}=(1-\hat{R}_{1}),$ where $\tilde{P}_{1}$ is the first size biased pick from $P,$ and $\hat{P}_{\alpha,\alpha}\sim \mathrm{PK}_{\alpha}(\hat{h}_{\alpha,\alpha}\cdot f_{\alpha},H).$ Although not immediately obvious, one may apply Theorem~\ref{thm:post1} to $\hat{P}_{\alpha,\alpha},$ in particular to $\hat{P}_{\alpha,\alpha}|\hat{R}_{1}=b\sim \mathrm{PK}_{\alpha}(h^{(b)}_{\alpha,\alpha}\cdot f_{\alpha},H)$ to $h^{(b)}_{\alpha,\alpha}$ in place of $h,$
with again $n=1$ to deduce a recursion leading to a stick-breaking representation, which can otherwise be deduced from~\cite{PPY92} under a change of measure for general $h$, albeit with perhaps a more transparent view of the dependence structure.

As we have mentioned, the representations in Theorem~\ref{thm:post2} are related to repeated applications of a fragmentation procedure by independent $\mathrm{PD}(\alpha,1-\alpha)$, originated in~\cite{DGM}, as expressed in more detail in~\cite[Section 3]{HJL2}. In particular, we have shown that probability masses of $\hat{P}_{\alpha,n}$ in ~\eqref{fpost2} has the conditional distribution of $(P_{k})\sim PK_{\alpha}(h\cdot f_{\alpha})$ fragmented $n$ independent times given $K_{n}=k$, as described in~\cite[eq. (3.1)]{HJL2}, since the distribution of $\hat{P}_{\alpha,n}$ coincides with that in~\cite[eq. (3.8)]{HJL2}, with $n$ in place of $r$.
 
\section{Examples}
We now present two examples. First for completeness we show how to recover the known results for the Pitman-Yor process as described in ~\cite[Corollary 20]{Pit96}. We then apply our Theorem to achieve new results for the Mittag-Leffler class which is related to the Pitman-Yor process by conditioning as described below. 

\subsection{Pitman-Yor processes}
Setting $h(t)=t^{-\theta}/\mathbb{E}[T^{-\theta}_{\alpha}]$ yields $P\sim\mathscr{PY}(\alpha,\theta,H)$, with EPPF
\begin{equation}
p_{\alpha,\theta}(n_{1},\ldots,n_{k})=\frac{\alpha^{k}\Gamma(\theta/\alpha+k)}{\Gamma(\theta/\alpha)}\frac{\Gamma(\theta)}{\Gamma(\theta+n)}\prod_{i=1}^{k} \frac{\Gamma(n_{i}-\alpha)}{\Gamma(1-\alpha)}.
\label{EPPFA}
\end{equation}
It follows that the joint density of the corresponding $\hat{R}_{k},\hat{T}_{\alpha,k\alpha},$ as described in~\eqref{jointRT} satisfies
$$
f_{B_{\theta+k\alpha,n-k\alpha}}(b)f_{\alpha,\theta+k\alpha}(t)\propto t^{-\theta}b^{\theta}f_{R_k}(b)f_{\alpha,k\alpha}(t).
$$

That is $\hat{R}_{k}\sim \Beta(\theta+k\alpha,n-k\alpha)$  independent of $\hat{T}_{\alpha,k\alpha}\overset{d}=T_{\alpha,\theta+k\alpha}.$ Where the latter implies that $\hat{P}_{\alpha,k\alpha}\overset{d}=P_{\alpha,\theta+k\alpha}\sim\mathscr{PY}(\alpha,\theta+k\alpha,H),$ which yields the result in~\cite[Corollary 20]{Pit96},
\begin{equation}\label{PYpost3}
\tilde{P}^{(n)}\overset{d}=B_{\theta+k\alpha,n-k\alpha}P_{\alpha,\theta+k\alpha} + (1 -B_{\theta+k\alpha,n-k\alpha}) \sum_{j=1}^{k} D_j \delta_{\tilde{X}_{j}},
\end{equation}
and similarly using Theorem~\ref{thm:post2}, $\hat{P}_{\alpha,n}\overset{d}=P_{\alpha,\theta+n}\sim\mathscr{PY}(\alpha,\theta+n,\mathbb{U}),$ independent of $\hat{\beta}_{k}\overset{d}=\beta_{\frac{\theta+k\alpha}{\alpha},\frac{n}{\alpha}-k}.$ To obtain
\begin{equation}\label{PYpost4}
\tilde{P}^{(n)}\overset{d}=P_{\alpha,\theta+n}\circ (B_{\frac{\theta}{\alpha}+k,\frac{n}{\alpha}-k}H+(1-B_{\frac{\theta}{\alpha}+k,\frac{n}{\alpha}-k}) \sum_{j=1}^{k} d_j \delta_{\tilde{X}_{j}}),
\end{equation}
Setting $n=1$ in~\eqref{PYpost3} and randomizing $\tilde{X}_{1}\sim H$ leads to the well-known representation~\cite{PPY92}, see also~\cite[Section 5]{JamesLamp} for references and background,

\begin{equation}\label{Stick2}
P_{\alpha,\theta}=B_{\theta+\alpha,1-\alpha}P_{\alpha,\theta+\alpha} + (1 -B_{\theta+\alpha,1-\alpha})\delta_{\tilde{X}_{1}}.
\end{equation}
Due to the special independence properties of the components and the fact that the result hold for all $\theta>-\alpha,$ allows one to obtain the well known stick-breaking representation due to~\cite{PPY92}, see also \cite{IJ2001,Pit96,PY97}, by applying~\eqref{Stick2} to $P_{\alpha,\theta+\alpha},$ then to $P_{\alpha,\theta+2\alpha}$ and so on.

\subsection{Posterior distributions of the generalized Mittag-Leffler class} As in~
\cite{HJL2,HJL3}  consider $(P_{l})\sim \mathrm{PD}(\alpha,\theta)$ with local time or $\alpha$-diversity $L_{\alpha,\theta}\overset{d}=T^{-\alpha}_{\alpha,\theta}\sim \mathrm{ML}(\alpha,\theta),$
 with density $g_{\beta,\theta}(s)=s^{\theta/\alpha}g_{\alpha}(s)/\mathbb{E}[T^{-\theta}_{\alpha}].$ It's Laplace transform as described in \cite[Section~3]{JamesLamp}, and \cite[Section 4.1]{HJL2}, see also ~\cite{GorenfloMittag} for further general background, is
\begin{equation}
\mathbb{E}\big[{\mbox
e}^{-\lambda T^{-\alpha}_{\alpha,\theta}}\big]=\mathrm{E}^{(\frac{\theta}{\alpha}+1)}_{\alpha,\theta+1}(-\lambda),
\label{MittagLeffler}
\end{equation}
where,
\begin{equation}
\mathrm{E}^{(\frac{\theta}{\alpha}+1)}_{\alpha,\theta+1}(-\lambda)=\sum_{\ell=0}^{\infty}\frac{{(-\lambda)}^{\ell}}{\ell!}\frac{
\Gamma(\frac{\theta}{\alpha}+1+\ell)\Gamma(\theta+1)}
{\Gamma(\frac{\theta}{\alpha}+1)\Gamma(\alpha
\ell+\theta+1)},\qquad
\theta>-\alpha,
\label{altM}
\end{equation}
and from~\cite[Proposition 4.4]{HJL2} there is the density
\begin{equation}
\label{tiltMLden}
g^{(0)}_{\alpha,\theta+j\alpha}(s|\lambda)=
\mathbb{P}(L_{\alpha,\theta}\in ds|N(\lambda L_{\alpha,\theta})=j)/ds=\frac{{\mbox e}^{-\lambda s}g_{\alpha,\theta+j\alpha}(s)}{\mathrm{E}^{(\frac{\theta}{\alpha}+j+1)}_{\alpha,\theta+j\alpha+1}(-\lambda)}
\end{equation}
where $(N(t),t\ge 0)$ denotes a standard Poisson process with mean intensity $\mathbb{E}[N(t)]=t,$ hence $N(\lambda L_{\alpha,\theta})$ is a mixed Poisson random variable. Then as in ~\cite{HJL2,HJL3} the conditional distribution of $(P_{l})|N(\lambda L_{\alpha,\theta})=j$ for $\lambda>0$ and fixed $j\in \{0,1,2,\ldots\}$ has the following distribution
\begin{equation}
\mathbb{L}^{(0)}_{\alpha,\theta+j\alpha}(\lambda):=\int_{0}^{\infty}\mathrm{PD}(\alpha|s^{-\frac1\alpha})\,
g^{(0)}_{\alpha,\theta+j\alpha}(s|\lambda)\,ds,
\label{genMittaglambda2}
\end{equation}
and, hence, in this case,
$$
h(t):=\vartheta^{(\lambda)}_{\alpha,\theta+j\alpha}(t)=\frac{{\mbox e}^{-\lambda t^{-\alpha}}t^{-\theta-j\alpha}}{\mathrm{E}^{(\frac{\theta}{\alpha}+j+1)}_{\alpha,\theta+j\alpha+1}(-\lambda)\mathbb{E}[T_{\alpha}^{-\theta-j\alpha}]}.
$$
Since $\theta>-\alpha$ is quite general we may focus on the case where $j=0,$ it is notable that we will encounter the case with $k$ in place of $j$ when describing the posterior distribution. Specifically we consider $(P_{v}(\lambda))\sim \mathbb{L}^{(0)}_{\alpha,\theta}(\lambda)$ which is equivalent to $\mathrm{PK}_{\alpha}(\vartheta^{(\lambda)}_{\alpha,\theta}\cdot f_{\alpha}).$

Furthermore, as in~\cite{HJL2}, set $Z^{\frac{n-k\alpha}{\alpha}}_{\alpha,\theta+k\alpha}=T^{-\alpha}_{\alpha,\theta+k\alpha}B^{\alpha}_{\theta+k\alpha,n-k\alpha}\overset{d}=T^{-\alpha}_{\alpha,\theta+n}B_{\frac{\theta}{\alpha}+k,\frac{n}{\alpha}-k}.$ 
Then from~\cite[Propositions 4.1 and 4.2]{HJL2},  
\begin{equation}
\mathbb{E}\big[{\mbox
e}^{-\lambda Z^{\frac{n-k\alpha}{\alpha}}_{\alpha,\theta+k\alpha}}\big]=\mathrm{E}^{(\frac{\theta}{\alpha}+k)}_{\alpha,\theta+n}(-\lambda),
\label{MittagLeffler2}
\end{equation}
where $\mathrm{E}^{(\frac{\theta}{\alpha}+k)}_{\alpha,\theta+n}(-\lambda)=\mathbb{E}\bigg[\mathrm{E}^{(\frac{\theta+n}{\alpha})}_{\alpha,\theta+n}\big(-\lambda {B_{\frac{\theta}{\alpha}+k,\frac{n}{\alpha}-k}}\big)\bigg]$ can be expressed as
$$
\mathrm{E}^{(\frac{\theta}{\alpha}+k)}_{\alpha,\theta+n}(-\lambda)=
\sum_{\ell=0}^{\infty}\frac{{(-\lambda)}^{\ell}}{\ell!}\frac{\Gamma(\frac{\theta}{\alpha}+k+\ell)\Gamma(
\theta+n)}{\Gamma(\frac{\theta}{\alpha}+k)\Gamma(\alpha
\ell+\theta+n)},
$$
The corresponding EPPF, as given in~\cite[Proposition 4.5]{HJL2}, can be expressed as 
$$
p^{(0)}_{\alpha,\theta}(n_{1},\ldots,n_{k}|\lambda)=\frac{\mathrm{E}^{(\frac{\theta}{\alpha}+k)}_{\alpha,\theta+n}(-\lambda)}{\mathrm{E}^{(\frac{\theta}{\alpha}+1)}_{\alpha,\theta+1}(-\lambda)}p_{\alpha,\theta}(n_{1},\ldots,n_{k})
$$
We say that $P=P^{(\lambda)}_{\alpha,\theta}:=\sum_{l=1}^{\infty}P_{l}(\lambda)\delta_{V_{l}}\sim \mathbb{L}_{\alpha,\theta}(\lambda,H):=\mathrm{PK}_{\alpha}(\vartheta^{(\lambda)}_{\alpha,\theta}\cdot f_{\alpha},H)$ meaning its distribution is equivalent to $P_{\alpha,\theta}|N(\lambda L_{\alpha,\theta})=0$ for $P_{\alpha,\theta}\sim \mathscr{PY}(\alpha,\theta,H).$ The corresponding prediction rule for $X_{n+1}|X_{1},\ldots,X_{n},$ is
$$
\frac{\mathrm{E}^{(\frac{\theta}{\alpha}+k+1)}_{\alpha,\theta+n+1}(-\lambda)}{\mathrm{E}^{(\frac{\theta}{\alpha}+k)}_{\alpha,\theta+n}(-\lambda)}\frac{\theta+k\alpha}{\theta+n}H+
\frac{\mathrm{E}^{(\frac{\theta}{\alpha}+k)}_{\alpha,\theta+n+1}(-\lambda)}{\mathrm{E}^{(\frac{\theta}{\alpha}+k)}_{\alpha,\theta+n}(-\lambda)}\sum_{j=1}^{k}\frac{n_{j}-\alpha}{\theta+n}\delta_{\tilde{X}_{j}}
$$

We now proceed to describe the posterior distribution and related quantities.
\begin{prop}\label{Mittagprop}Consider the results in Theorem~\ref{thm:post1} In this case for $P$ specified as  $P^{(\lambda)}_{\alpha,\theta}:=\sum_{l=1}^{\infty}P_{l}(\lambda)\delta_{V_{l}}\sim \mathbb{L}_{\alpha,\theta}(\lambda,H):=\mathrm{PK}_{\alpha}(\vartheta^{(\lambda)}_{\alpha,\theta}\cdot f_{\alpha},H)$ the posterior distribution of $P^{(\lambda)}_{\alpha,\theta}|X_{1},\ldots,X_{n}$ can be expressed as
$$
R_{\alpha,\theta+k\alpha}(\lambda)P^{(\lambda R^{\alpha}
_{\alpha,\theta+k\alpha}(\lambda))}_{\alpha,\theta+k\alpha}+(1-R_{\alpha,\theta+k\alpha}(\lambda))\sum_{j=1}^{k} D_j \delta_{\tilde{X}_{j}},
$$
\begin{enumerate}
\item $(D_1, \ldots, D_k) \sim \mathrm{Dir}(n_1 - \alpha, \ldots, n_k - \alpha)$ is independent of $R_{\alpha,\theta+k\alpha}(\lambda)$ and $P^{(\lambda R^{\alpha}.
_{\alpha,\theta+k\alpha}(\lambda))}_{\alpha,\theta+k\alpha}.$
\item $P^{(\lambda R^{\alpha}
_{\alpha,\theta+k\alpha}(\lambda))}_{\alpha,\theta+k\alpha}|R_{\alpha,\theta+k\alpha}(\lambda)=b\sim \mathbb{L}_{\alpha,\theta+k\alpha}(\lambda b^{\alpha},H):=\mathrm{PK}_{\alpha}(\vartheta^{(\lambda b^{\alpha})}_{\alpha,\theta+k\alpha}\cdot f_{\alpha},H),$ with local time at $1$ or $
\alpha$-diversity having density
$g^{(0)}_{\alpha,\theta+k\alpha}(s|\lambda b^{\alpha})$ as specified in ~\eqref{tiltMLden}.
\item $R_{\alpha,\theta+k\alpha}(\lambda)$ has density
$$
\frac{\mathrm{E}^{(\frac{\theta}{\alpha}+k)}_{\alpha,\theta+k\alpha}(-\lambda b^{\alpha})}{\mathrm{E}^{(\frac{\theta}{\alpha}+k)}_{\alpha,\theta+n}(-\lambda)}f_{B_{\theta+k\alpha,n-k\alpha}}(b)
$$
\end{enumerate}
\end{prop}
\begin{proof}Checking Theorem~\ref{thm:post1} with $h(t)=\vartheta^{(\lambda)}_{\alpha,\theta}(t),$ one has from~\eqref{jointRT} that in this case
$$
f_{\hat{R}_k,\hat{T}_{\alpha,k\alpha}}(b,t) \propto {\mbox e}^{-\lambda t^{-\alpha}b^{\alpha}}t^{-\theta}b^{\theta}f_{R_k}(b)f_{\alpha,k\alpha}(t),
$$
hence 
$
f_{\hat{R}_k,\hat{T}_{\alpha,k\alpha}}(b,t) \propto {\mbox e}^{-\lambda t^{-\alpha}b^{\alpha}}f_{B_{\theta+k\alpha,n-k\alpha}}(b)f_{\alpha,\theta+k\alpha}(t).
$
Hence one can deduce from \eqref{MittagLeffler2} that one can set 
$\hat{R}_{k}=R_{\alpha,\theta+k\alpha}(\lambda)$ with marginal density as specified. Similarly it follows that $\hat{T}_{\alpha,k\alpha}| \hat{R}_{k}=b$ has the density $g^{(0)}_{\alpha,\theta+k\alpha}(s|\lambda b^{\alpha}).$ Other details of the result are then provided by 
Theorem~\ref{thm:post1} and the definition of $\mathbb{L}_{\alpha,\theta+k\alpha}(\lambda b^{\alpha},H).$
\end{proof}
Similar to~\eqref{PYpost4}, Proposition~\ref{Mittagprop} yields
\begin{equation}\label{MittagStick2}
P^{(\lambda) }_{\alpha,\theta}=
R_{\alpha,\theta+\alpha}(\lambda)P^{(\lambda R^{\alpha}
_{\alpha,\theta+\alpha}(\lambda))}_{\alpha,\theta+\alpha}+(1-R_{\alpha,\theta+\alpha}(\lambda))\delta_{\tilde{X}_{1}}.
\end{equation}
where $(1-R_{\alpha,\theta+\alpha}(\lambda))$ is the first size-biased pick from $P^{(\lambda) }_{\alpha,\theta}.$
While $R_{\alpha,\theta+\alpha}(\lambda)$ and $P^{(\lambda R^{\alpha}
_{\alpha,\theta+\alpha}(\lambda))}_{\alpha,\theta+\alpha}$ in~\eqref{MittagStick2} are not independent, one may develop a recursive formula by using $P^{(\lambda R^{\alpha}
_{\alpha,\theta+\alpha}(\lambda))}_{\alpha,\theta+\alpha}|R^{\alpha}
_{\alpha,\theta+\alpha}(\lambda)=b_{1}$, which is $P^{(\lambda b^{\alpha}_{1})}_{\alpha,\theta+\alpha}$ and hence  in the same family of distributions as $P^{(\lambda) }_{\alpha,\theta}$, with $\lambda b^{\alpha}_{1}$ in place of $\lambda$. Iterating this, yields a stick-breaking representation. We leave the details to the interested reader. 

The proof of the next results follows in a similar fashion as for Proposition~\ref{Mittagprop} using Theorem~ \ref{thm:post2} and the definitions above. We omit the details which are now fairly straightforward.  
\begin{prop}Consider the results in Theorem~ \ref{thm:post2}. In this case for $P$ specified as  $P^{(\lambda)}_{\alpha,\theta}:=\sum_{l=1}^{\infty}P_{l}(\lambda)\delta_{V_{l}}\sim \mathbb{L}_{\alpha,\theta}(\lambda,H):=\mathrm{PK}_{\alpha}(\vartheta^{(\lambda)}_{\alpha,\theta}\cdot f_{\alpha},H)$ the posterior distribution of $P^{(\lambda)}_{\alpha,\theta}|X_{1},\ldots,X_{n}$ can be expressed as
$$
P^{(\lambda\beta_{
\frac{\theta}{\alpha}+k,\frac{n}{\alpha}-k}(\lambda))}_{\alpha,\theta+n}\circ (\beta_{
\frac{\theta}{\alpha}+k,\frac{n}{\alpha}-k}(\lambda)H+(1-\beta_{\frac{\theta}{\alpha}+k,\frac{n}{\alpha}-k}(\lambda))\sum_{j=1}^{k}d_j\delta_{\tilde{X}_{j}})
$$
\begin{enumerate}
\item $(d_1, \ldots, d_k) \sim \mathrm{Dir}(\frac{n_1 - \alpha}{\alpha}, \ldots, \frac{n_k - \alpha}{\alpha})$ is independent of  $\beta_{
\frac{\theta}{\alpha}+k,\frac{n}{\alpha}-k}(\lambda)$ and $P^{(\lambda\beta_{
\frac{\theta}{\alpha}+k,\frac{n}{\alpha}-k}(\lambda))}_{\alpha,\theta+n}.$
\item $\beta_{
\frac{\theta}{\alpha}+k,\frac{n}{\alpha}-k}(\lambda)$ has density
$$
\frac{\mathrm{E}^{(\frac{\theta+n}{\alpha})}_{\alpha,\theta+n}(-\lambda b)}{\mathrm{E}^{(\frac{\theta}{\alpha}+k)}_{\alpha,\theta+n}(-\lambda)}f_{B_{
\frac{\theta}{\alpha}+k,\frac{n}{\alpha}-k}}(b)
$$
\item $P^{(\lambda\beta_{
\frac{\theta}{\alpha}+k,\frac{n}{\alpha}-k}(\lambda))}_{\alpha,\theta+n}|\beta_{
\frac{\theta}{\alpha}+k,\frac{n}{\alpha}-k}(\lambda)=b\sim \mathbb{L}_{\alpha,\theta+n}(\lambda b,\mathbb{U}):=\mathrm{PK}_{\alpha}(\vartheta^{(\lambda b)}_{\alpha,\theta+n}\cdot f_{\alpha},\mathbb{U}),$with local time at $1$ or $
\alpha$-diversity having density
$g^{(0)}_{\alpha,\theta+n}(s|\lambda b)$ as specified in ~\eqref{tiltMLden}.
\end{enumerate}
\end{prop}
The next result states that the marginal distribution $P^{(\lambda\beta_{
\frac{\theta}{\alpha}+k,\frac{n}{\alpha}-k}(\lambda))}_{\alpha,\theta+n}$ belongs to a family of distributions described in \cite[ eq. (4.13) and (4.14)]{HJL2}. Which follows from a direct comparison.

\begin{prop}
The marginal distribution of  $P^{(\lambda\beta_{
\frac{\theta}{\alpha}+k,\frac{n}{\alpha}-k}(\lambda))}_{\alpha,\theta+n}\sim\mathbb{L}^{(n,k)}_{\alpha,\theta}(\lambda,\mathbb{U}),$ where the law of its ranked masses $(Q_{v}(\lambda))$ is 
$$
\mathbb{L}^{(n,k)}_{\alpha,\theta}(\lambda)=\int_{0}^{\infty}\mathrm{PD}(\alpha|s^{-\frac{1}{\alpha}})\,g^{(n,k)}_{\alpha,\theta}(s|\lambda)\,ds
$$
with , for $g_{\alpha,\theta+n}(s)$ the density of $T^{-\alpha}_{\alpha,\theta+n}\sim \mathrm{ML}(\alpha,\theta+n)$,
$$
g^{(n,k)}_{\alpha,\theta}(s|\lambda)=\frac{_1F_1(\frac{\theta}{\alpha}+k;\frac{\theta}{\alpha}+\frac{n}{\alpha};-\lambda s)}{\mathrm{E}^{(\frac{\theta}{\alpha}+k)}_{\alpha,\theta+n}(-\lambda)}g_{\alpha,\theta+n}(s)
$$
where, 
$
\mathbb{E}\left[{\mbox e}^{-\lambda B_{\frac{\theta}{\alpha}+ k,\frac{n}{\alpha}-k}}\right]
= {_1F_1\left(\frac{\theta}{\alpha}+k;\frac{\theta}{\alpha}+\frac{n}{\alpha};-\lambda \right)}
$
is a confluent hypergeometric function of the first kind.
\end{prop}

\begin{rem}
We considered demonstrating our results for the family $\mathbb{P}^{[j]}(\lambda)=\int_{0}^{\infty}\mathrm{PD}(\alpha|t)f^{[j]}{\alpha}(t|\lambda) dt$, which consists of size-biased generalized gamma processes, where $f^{[j]}(t|\lambda)\propto\lambda^{j}{\mbox e}^{-\lambda}f_{\alpha}(t)$ for $j=0,1,2,\ldots$. However, there is considerable overhead in notation, and obtaining the finest results requires further developments that are available to us but may distract from the main developments presented here. See~\cite{JamesStick,JamesFragblock} for more on this class.
\end{rem}

\begin{funding}
L.F. James was supported in part by grants RGC-GRF 16301521, of the Research Grants Council (RGC) of the Hong Kong SAR. 
\end{funding}


\end{document}